\documentclass[11pt]{amsart}

\usepackage[alphabetic, initials]{amsrefs}
\usepackage{amssymb}
\usepackage{graphicx}
\usepackage{hyperref}

\calclayout

\numberwithin{equation}{section}

\theoremstyle{definition}
\newtheorem{theorem}{Theorem}[section]
\newtheorem{definition}[theorem]{Definition}

\newtheorem{proposition}[theorem]{Proposition}
\newtheorem{lemma}[theorem]{Lemma}

\newcommand{\R}{\mathbb{R}}

\DeclareMathOperator{\vol}{Vol}

\DeclareMathOperator{\lsp}{span}

\begin{document}


\title{Counting Ancient Solutions on A Strip with Exponential Growth}
\author{Feng Gui}
\address{Department of Mathematics, Massachusetts Institute of Technology, Cambridge, MA 02139, USA}
\email{fenggui@mit.edu}

\begin{abstract}
  We study the ancient solutions of parabolic equations on an infinite strip. We show that any polynomial growth ancient solution for a class of parabolic equations must be constant. Furthermore, we show that the vector space of ancient solutions that grow slower than a fixed exponential order is of finite dimension.
\end{abstract}

\maketitle


\section{Introduction}

It has been a common theme that the function theory on a manifold usually reflects the geometric properties. One of the most well-known examples is the classical study of meromorphic functions on Riemann Surfaces. In this paper, we are interested in the ancient caloric functions, that is, solutions of the heat equation that exist for all past time. We show that any nonzero ancient caloric function on an infinite strip has at least exponential growth. This is different from what happens on a Euclidean space or on a complete manifold with nonnegative Ricci curvature where ancient solutions with polynomial growth exist. Furthermore, we bound the dimension of ancient caloric functions with exponential growth on a strip in a similar way to how the dimension of polynomial growth harmonic functions is bounded on a Euclidean space.

Let $\Omega = \R \times \Omega_0$ be an infinite strip where $\Omega_0 \subset \R^n$ is an open bounded set with volume $V_0$. We will use the standard coordinates $x = (x_0,...,x_n)$ on $\R^{n+1}$ where $x_0$ axis is the direction that the strip goes to infinity. Consider a uniformly elliptic operator $L$ in the divergence form, that is,
\begin{equation}
  Lu = \partial_i(a^{ij} \partial_j u) + b^i \partial_i u + cu,
\end{equation}
where $a^{ij}(x)$, $b^i(x)$ and $c(x)$ are bounded on $\Omega$ and there exist $0 < \lambda < \Lambda$ such that
\begin{equation}\label{ellipticity}
  \lambda |\xi|^2 \leq a^{ij}(x)\xi_i\xi_j \text{ and } |a^{ij}(x)| \leq \Lambda
\end{equation}
for all $x\in \Omega$ and $\xi \in \R^{n+1}$. The Einstein summation convention is applied here and henceafter. We also need to assume some smallness condition in the first and zeroth order terms, that is, for some $\varepsilon>0$,
\begin{equation}\label{smallness}
  |b^i(x)|\leq \sqrt{\varepsilon}\text{ and }c(x) \leq \varepsilon.
\end{equation}
The small constant $\varepsilon$ will be determined depending on $n$, $\lambda$ and $V_0$ in Lemma \ref{Lem: rev Poincare}. The Laplacian of course satisfies this condition and it is the motivating example for our study.

We consider the solutions $u(t,x)$ to the following problem:
\begin{equation}\label{eqn1}
  \begin{cases}
    (\partial_t - L)u = 0 & \text{for } (t,x)\in \R^-\times \Omega \\
    u = 0 & \text{for } (t,x)\in \R^-\times \partial \Omega
  \end{cases}.
\end{equation}
where we denote $\R^- = (-\infty, 0]$. Such solution is called an ancient solutions as it is defined for all negative time. For $d\geq 1$, we define the solution spaces of polynomial and exponential growth
\begin{definition}\label{def: solution space}
  \begin{equation}
    \begin{aligned}
      P_d := \{&u(t,x)\in W^{1,2}_{loc}(\R^-\times \Omega), u \text{ solves equation }(\ref{eqn1}) \text{ and}\\
               &\text{ there exists } C = C(u)>0 \text{ s.t. } |u(t,x)| \leq C(|x|+|t|^\frac{1}{2}+1)^d\},
    \end{aligned}
  \end{equation}
  \begin{equation}
    \begin{aligned}
      E_d := \{&u(t,x)\in W^{1,2}_{loc}(\R^-\times \Omega), u \text{ solves equation }(\ref{eqn1}) \text{ and}\\
               &\text{ there exists } C = C(u)>0 \text{ s.t. } |u(t,x)| \leq Ce^{d(|x|+|t|^\frac{1}{2})}\}.
    \end{aligned}
  \end{equation}
\end{definition}

Our main results are the following.
\begin{theorem}\label{main thm 1}
  If $L$ satisfies conditions (\ref{ellipticity}) and (\ref{smallness}) and $d>0$, then $P_d = \{0\}$.
\end{theorem}
\begin{theorem}\label{main thm 2}
  If $L$ satisfies conditions (\ref{ellipticity}) and (\ref{smallness}) and $d\geq 1$, then $\dim E_d \leq C d^{n+2}$ for some constant $C = C(n, \lambda, \Lambda, V_0)$.
\end{theorem}

The study of solutions with certain growth conditions originates from the classical Liouville theorem, which states that positive or bounded harmonic functions on $\R^n$ are constants. The Liouville theorem was generalized to manifolds with nonnegative Ricci curvature in \cites{Yau75, ChY75} through gradient estimate. Yau then conjectured that on a manifold with nonnegative Ricci curvature, the space of harmonic functions with polynomial growth is finite dimensional. The conjecture was resolved in \cite{CM97}. In fact, Colding and Minicozzi proved the finite dimensionality for a larger family of metric measure spaces where volume doubling and Neumann-Poincare inequality hold. \cite{CM98} showed that dimension is bounded by $Cd^{n-1}$ and the exponent $n-1$ is optimal in the case of $\R^n$. When the domain is an infinite strip instead of the Euclidean space, Hang and Lin, \cite{HL99}, showed that the harmonic function with polynomial growth has to be zero. Using Colding-Minicozzi's construction \cite{CM97, CM98}, they further showed that there are only finitely many linearly independent harmonic functions with exponential polynomial growth. Recently, \cite{H20} refined their bound.

On the parabolic side of the story, the key ingredient is Li-Yau's parabolic gradient estimate \cite{LY86}, replacing Cheng-Yau's estimate for the harmonic functions. Souplet and Zhang proved an analog of Liouville theorem for the heat equation on manifold in \cite{SZ06}, stating that any sublinear eternal caloric function on a noncompact manifold with nonnegative Ricci curvature has to be a constant. However, when counting the dimension of the solution space with growth condition, the ancient caloric functions behave more like harmonic functions compared to the eternal ones. In recent years, \cite{LZ19} bounded the dimension of ancient caloric functions with polynomial growth of degree less than or equal to $d$ on a $n$-dimensional manifold with nonnegative Ricci curvature by $Cd^{n+1}$. Colding and Minicozzi proved a sharper bound $Cd^n$ in \cite{CM19} which has optimal dependence of $d$ in the case of $\R^n$. In \cite{H19}, Hua generalized their result to graphs. Using the construction in \cite{CM97, CM98} and a new localization inequality, Colding and Minicozzi bounded the dimension of ancient caloric functions on mean curvature flow by its entropy in \cite{CM20}. This gives a bound to the codimension of mean curvature flow.

Theorem \ref{main thm 1} and \ref{main thm 2} are the parabolic analogs of Han and Lin's results in \cite{HL99}. One could further study the solution space of noncompact manifolds with multiple ends and certain volume growth, gaining more understanding of how the geometry of the manifold changes the growth rate of harmonic and caloric functions on it. However, we should note that polynomial growth harmonic functions do exist on some noncompact manifolds with two ends, such as catenoids. This hints that some geometric conditions are necessary to control the growth of the harmonic and ancient caloric functions.

We should remark that the parabolic scaling ratio between $|x|$ and $|t|$ in the growth condition $e^{d(|x|+|t|^\frac{1}{2})}$ in Definition \ref{def: solution space} is crucial and cannot be relaxed to $e^{d(|x|+|t|)}$. Otherwise, one has no hope to prove any finite dimensionality of the solution space since separation of variables method already gives infinitely many linearly independent ancient heat solutions on $[0,1]\times \R$.

The paper will be structured as follows. In section 2, we show that ancient solutions to (\ref{eqn1}) satisfy a reverse Poincar\'e type inequality and thus cannot grow slower than any polynomials. In section 3, we recall some basic properties of exponential growth monotone functions and construct ``good'' solutions from the given ones. In section 4, we prove the main theorem \ref{main thm 2}.

{\bf Acknowledgements.} I would like to thank my advisor Professor William Minicozzi for his patient guidance and constant support.

\section{Growth of ancient solutions}
Let $D_r = (-r, r)\times \Omega_0 \subset \Omega$ and $Q_r = (-r^2, 0] \times D_r$ be our version of parabolic ``ball'' in the infinite strip. The key estimate that controls the growth of the solution is the following reverse Poincar\'e type inequality.

\begin{lemma}\label{Lem: rev Poincare}
  There exists constant $\varepsilon_0 = \varepsilon_0(n, \lambda, V_0) > 0$ such that if condition (\ref{smallness}) holds for $\varepsilon = \varepsilon_0$, then for any solution $u\in W^{1,2}_{loc}(\R^-\times \Omega)$ of equation (\ref{eqn1}) and $0< r < R$, there is a constant $C = C(\lambda, \Lambda)$ such that
  \begin{equation}
    \int_{Q_r} |\nabla u|^2 \leq \frac{C}{(R-r)^2} \int_{Q_R\setminus Q_r} u^2.
  \end{equation}
  Furthermore, there exists constant $C = C(n, \lambda, \Lambda, V_0)$ such that
  \begin{equation}
    \int_{Q_r} u^2 \leq \frac{C}{(R-r)^2} \int_{Q_R\setminus Q_r} u^2.
  \end{equation}
\end{lemma}

\begin{proof}
  Let $\varphi = \varphi(t,x_0) \in C^\infty_c(\R^- \times \R)$ be a compactly supported cutoff function in both time direction and the direction where the strip goes to infinity. With a slight abuse of notation, we will write $\varphi(t, x) = \varphi(t, x_0)$ for $x\in \R^{n+1}$. Using integration by parts, we have
  \begin{equation}\label{eqn: 0}
    \begin{aligned}
      \frac{1}{2}\partial_t \int_{\Omega} u^2\varphi^2
      &= \int_{\Omega} \varphi^2 uLu + \int_{\Omega} u^2 \varphi \partial_t \varphi \\
      &= \int_{\Omega} \varphi^2u(\partial_i(a^{ij} \partial_j u) + b^i\partial_i u + cu) +\int_{\Omega} u^2 \varphi \partial_t \varphi\\
      &=  -\int_\Omega (\varphi^2a^{ij}\partial_i u \partial_j u + 2\varphi u a^{ij}\partial_i \varphi \partial_j u) \\
      &\quad +\int_\Omega \varphi^2u(b^i\partial_i u + cu) + \int_{\Omega} u^2 \varphi \partial_t \varphi.
    \end{aligned}
  \end{equation}
  From condition (\ref{ellipticity}), we bound
  \begin{equation}\label{eqn: 1}
    \begin{aligned}
      \int_\Omega \varphi^2 a^{ij}\partial_i u \partial_j u + \int_\Omega 2\varphi u a^{ij}\partial_i \varphi \partial_j u
      &\geq \lambda \int_\Omega \varphi^2 |\nabla u|^2 - 2\Lambda\int_\Omega |\varphi||u||\partial_0 \varphi||\nabla u| \\
      & \geq \frac{3\lambda}{4} \int_\Omega \varphi^2 |\nabla u|^2 - C\int_\Omega u^2 |\partial_0 \varphi|^2
    \end{aligned}
  \end{equation}
  where $C  = 4\lambda^{-1} \Lambda^2 = C(\lambda, \Lambda)$ using inequality $\eta a^2+ \eta^{-1}b^2 \geq 2ab$ and $\partial_0 \varphi = \frac{\partial \varphi}{\partial x_0}$.
  Assume that condition (\ref{smallness}) holds for some $\varepsilon$, then we can bound the lower order terms
  \begin{equation}\label{eqn: 2}
    \begin{aligned}
      \int_\Omega \varphi^2u(b^i\partial_i u+ cu)
      &\leq \int_\Omega \sqrt{\varepsilon} \varphi^2 |u||\nabla u| + \varepsilon\int_\Omega\varphi^2 u^2 \\
      &\leq \frac{\lambda}{4}\int_\Omega \varphi^2|\nabla u|^2 + C\varepsilon \int_\Omega \varphi^2 u^2
    \end{aligned}
  \end{equation}
  where $C = C(\lambda)$. Since $\varphi$ is constant on each slice of the $\Omega$, applying Poincar\'e inequality, we have that
  \begin{equation}\label{eqn: 3}
      \int_\Omega \varphi^2 u^2 = \int_\R \varphi^2 \int_{\{x_0\}\times \Omega_0} u^2 \leq C\int_\R \varphi^2 \int_{\{x_0\}\times \Omega_0} |\nabla^{\Omega_0} u|^2 \leq C \int_\Omega \varphi^2 |\nabla u|^2
  \end{equation}
  where $C= C(n, V_0)$. Here, $\nabla^{\Omega_0} u$ is the projection of $\nabla u$ to its last $n$ coordinates and thus has smaller norm. From equations (\ref{eqn: 2}) and (\ref{eqn: 3}), we can take $\varepsilon_0 = \varepsilon_0(n, \lambda, V_0)$ small so that
  \begin{equation}\label{eqn: 4}
    \int_\Omega \varphi^2u(b^i\partial_i u+ cu) \leq \frac{\lambda}{2} \int_\Omega \varphi^2 |\nabla u|^2.
  \end{equation}

  Combining (\ref{eqn: 0}), (\ref{eqn: 1}) and (\ref{eqn: 4}),
  \begin{equation}\label{eqn: 5}
      \partial_t \int_{\Omega} u^2\varphi^2 + \frac{\lambda}{2} \int_\Omega \varphi^2 |\nabla u|^2 \leq C\int_\Omega u^2|\partial_0 \varphi|^2 + 2\int_\Omega u^2 |\varphi| |\partial_t \varphi|.
  \end{equation}
  Integrate over $t\in \R^-$, we have
  \begin{align}\label{rev poincare}
      \int_{\{t = 0\} \times \Omega} u^2\varphi^2 + \frac{\lambda}{2}\int_{\R^- \times \Omega} \varphi^2 |\nabla u|^2
      \leq C\int_{R^-\times \Omega } u^2 |\partial_0 \varphi|^2 + 2 \int_{ \R^-\times \Omega} u^2 |\varphi||\partial_t \varphi|.
  \end{align}

  Now we specify the cutoff function $\varphi$. Let $\varphi = 1$ on $Q_r$ and $\varphi=0$ outside $Q_R$ and $\varphi \in [0,1]$. We can also let
  \begin{equation}\label{eqn: cutoff}
    |\partial_0 \varphi| \leq \frac{2}{R-r} \text{  and  } |\partial_t \varphi| \leq \frac{2}{R^2 - r^2} \leq \frac{2}{(R-r)^2}.
  \end{equation}
  Note that the first term on the left hand side in (\ref{rev poincare}) is positive. Combining (\ref{rev poincare}) and (\ref{eqn: cutoff}), we have
  \begin{equation}
    \int_{Q_r} |\nabla u|^2 \leq \frac{C}{(R-r)^2} \int_{Q_R\setminus Q_r} u^2
  \end{equation}
  where $C = C(\lambda, \Lambda)$ as desired. The second claim follows by applying Poincar\'e inequality to each slice of the strip.
\end{proof}

The lemma effectively lower bounds the growth of the ancient solution to equation (\ref{eqn1}). In fact, we show that any ancient solution with polynomial growth has to be zero.

\begin{proof}[Proof of Theorem \ref{main thm 1}]
Fix $r > 1$. For any $R > r$, we have
\begin{equation}\label{eqn: growth est}
  \int_{Q_R} u^2 \geq \left(1+\frac{(R-r)^2}{C}\right)\int_{Q_r} u^2
\end{equation}
for $C= C(n, \lambda, \Lambda, V_0)$ in Lemma \ref{Lem: rev Poincare}. Choose $r_0 = \sqrt{C(e-1)} > 0$. Iterations of inequality (\ref{eqn: growth est}) imply that for any integer $k>0$,
\begin{equation}
  \int_{Q_{r+kr_0}} u^2 \geq e^k \int_{Q_r} u^2.
\end{equation}
Taking $k$ to infinity, we have
\begin{equation}
  \liminf_{R\to\infty} R^{-1}\log \left(\int_{Q_R}u^2\right) \geq \frac{1}{r_0}
\end{equation}

If $u \in P_d$ is an ancient solution with polynomial growth, then $\int_{Q_R} u^2 \leq C R^{2d}$ for any $R > 1$. In particular,
\begin{equation}
  \int_{Q_r} u^2 \leq \frac{C(r+kr_0)^{2d}}{e^k}
\end{equation}
which leads to a contradiction for large $k$ unless $u\equiv 0$.
\end{proof}

\section{Construction of good solutions from given ones}
Theorem \ref{main thm 1} motivates us to study the ancient solutions with exponential growth instead. The proof of Theorem \ref{main thm 2} follows the idea in \cite{CM97} and \cite{CM98}, where they bound the dimension using a set of solutions with good properties. The exponential growth bound makes it possible to construct solutions with good properties from any given ones.

We start with some definitions (cf. \cite{CM97} section 2 and section 4). For each $r>0$, we define a positive semi-definite bilinear form on $L^2_{loc}(R^- \times \Omega)$
\begin{equation}
  J_r(u,v) = \int_{Q_r} uv
\end{equation}
for any locally square integrable functions $u,v$. We should note that $J_r$ is an inner product on $L^2(Q_r)$. When the unique continuation property holds for $u\in E_d$, for example when $L$ has Lipschitz coefficients, then $J_r$ defines an inner product on $E_d$. We also write
\begin{equation}
  I_u(r) = J_r(u,u) = \int_{Q_r} u^2
\end{equation}
for locally integrable function $u$. By definition, $I_u(r)$ is non-decreasing in $r$.

We say that a set of linearly independent functions $v_1, ...,v_\ell \in L^2_{loc}(\R^- \times \Omega)$ has good properties if, for some $r, \delta$ and $\sigma$, they are orthonormal with respect to $J_r$ and $I_{v_i}(r+\delta) \leq \sigma$. We will see in Proposition \ref{prop: construction} that it is possible to construct functions with good properties from any given set of linearly independent functions in $E_d$.

\begin{definition}\label{def: f_i}
  Let $u_1,...,u_k$ be linearly independent functions in $L^2_{loc}(R^- \times \Omega)$. For $r > 0$, we define
\begin{equation}
    \tilde{w}_{1,r} = u_1|_{Q_r}.
\end{equation}
   We define $\tilde{w}_{i, r}\in L^2(Q_r)$ inductively for $i =2,...,k$ as follows. Since $J_r$ defines an inner product on $L^2(Q_r)$, we can write $P_i$ as the orthogonal projection map onto the space $\lsp\{\tilde{w}_{1,r}, ..., \tilde{w}_{i-1, r} \}^\perp \subset L^2(Q_r)$ under the inner product $J_r$. Then,
  \begin{equation}
    \tilde{w}_{i,r} = P_i (u_i|_{Q_r})
  \end{equation}
  is well-defined and we can write
  \begin{equation}
    u_i|_{Q_r} = \sum_{j=1}^{i-1} \lambda_{ij}(r) u_j|_{Q_r} + \tilde{w}_{i, r}
  \end{equation}
  for some real coefficients $\lambda_{ij}(r)$. Moreover, we extend $w_{i,r}$ to a locally square integrable function on $R^- \times \Omega$ by
  \begin{equation}
    w_{i,r} = u_i - \sum_{j=1}^{i-1} \lambda_{ij}(r) u_j.
  \end{equation}

  We should remark that the restrictions $u_i|_{Q_r}$ are not necessarily linearly independent and thus $\lambda_{ij}(r)$ and $w_{i,r}$ are not uniquely defined. However, since $u_i$ are linearly independent, there exists $R>0$ such that $u_i|_{Q_r}$ are linearly independent for all $r> R$.

  We also define
  \begin{equation}\label{eqn: f_i}
    f_i(r) = \int_{Q_r} w_{i,r}^2 = \int_{Q_r} \tilde{w}_{i,r}^2.
  \end{equation}
\end{definition}

Here are a few immediate properties of $f_i(r)$.
\begin{lemma}\label{lem: properties of f}
  Given a set of linearly independent exponential growth ancient solutions $u_1, ..., u_k \in E_d$, define $f_i$ as in (\ref{eqn: f_i}). The following holds for $i = 1,...,k$:
  \begin{enumerate}
    \item $f_i(r) \leq C e^{(4d+1)r}$ where $C= C(u_i, V_0)$,
    \item $f_i(r) = I_{w_{i,r}}(r)$ and $f_i(s) \leq I_{w_{i,r}}(s)$ for any $s>0$, and
    \item $f_i$ is non-decreasing and nonnegative and is not identically zero.
  \end{enumerate}
\end{lemma}

\begin{proof}
  Since $w_{i,r}$ is the $J_r$-orthogonal projection of $u_i$,
  \begin{equation}
    \begin{aligned}
      f_i(r) \leq \int_{Q_r} u_i^2 \leq C \vol(Q_r)\, e^{4dr} \leq C'e^{(4d+1)r}
    \end{aligned}
  \end{equation}
  and (1) is thus established. For the same reason,
  \begin{equation}
    f_i(r) = \int_{Q_r} w_{i,r}^2 = \min \int_{Q_r} \big|u_i - \sum_{j=1}^{i-1} \lambda_j u_j\big|^2
  \end{equation}
  where the minimum is taken for all real coefficients $\lambda_j$. In particular, if we let $\lambda_j = \lambda_{ij}(s)$, then this implies $f_i(r) \leq I_{w_{i,s}}(r)$ for any $s>0$, which is the claim (2).

  If $s<r$, then $f_i(s) \leq I_{w_{i,r}}(s) \leq I_{w_{i,r}}(r) = f_i(r)$ by (2) and thus $f_i$ is monotone. Since $u_i|_{Q_r}$ are linearly independent for some large $r$, $w_{i,r}$ is not identically zero on $Q_r$. Claim (3) then follows.
\end{proof}

To construct solutions with good properties, we need the following lemma for functions with exponential growth bound.
\begin{lemma}\label{lem: exp growth}
  Let $f_1,..., f_k:[0,\infty)\to [0,\infty)$ be $k$ non-decreasing functions that are not identically zero. Assume that for some $d, C > 0$ and all $i = 1,...,k$,
  \begin{equation}\label{eqn: f_i exp bound}
    f_i(r) \leq Ce^{dr}.
  \end{equation}
  For any $\delta > 0$, $\ell < k$ and $\sigma > e^{\frac{k}{k-\ell +1}\delta d}$, there exist $\ell$ of these functions $f_{\alpha_1},..., f_{\alpha_\ell}$ and infinitely many integers $m$ such that, for $i = 1,...,\ell$,
  \begin{equation}\label{eqn: growth}
    f_{\alpha_i}((m+1)\delta) \leq \sigma f_{\alpha_i}(m\delta).
  \end{equation}
\end{lemma}
\begin{proof}
  We will show that for infinitely many integers $m$, (\ref{eqn: growth}) holds for some rank $\ell$ subset of $\{f_1,...,f_k\}$. Since the number of such subsets of $\{f_1,...,f_k\}$ is finite, there will be infinitely many integers $m$ such that (\ref{eqn: growth}) holds for some fixed rank $\ell$ subset.

  Let
  \begin{equation}
  f(r) = \prod_{i=1}^k f_i(r).
  \end{equation}
  By definition, $f$ is a nonnegative non-decreasing function on $[0,\infty)$ that is not identically zero. There exists $R>0$ such that $f(r)$ is strictly postive when $r> R$. From the bound (\ref{eqn: f_i exp bound}), we have
  \begin{equation}
    f(r) \leq C^ke^{kdr}.
  \end{equation}

  Now, for a contradiction, suppose there are finitely many integers $m$ such that (\ref{eqn: growth}) holds for some rank $\ell$ subset of $\{f_1,...,f_k\}$ and let $m_0$ be the largest one. For any integer $m > \max\{m_0, R\delta^{-1}\}$, we have $f(m\delta)>0$ and
  \begin{equation}
    f_i((m+1)\delta) > \sigma f_i (m\delta)
  \end{equation}
  for at least $k-\ell+1$ integers $i\in\{1,...,k\}$. Therefore
  \begin{equation}\label{eqn: basic step}
    f((m+1)\delta) > \sigma^{k-\ell+1}f(m\delta).
  \end{equation}
  Now fix an integer $m > \max\{m_0, R\delta^{-1}\}$. By iterations of (\ref{eqn: basic step}),
   \begin{equation}
     0< f(m\delta) < \frac{f((m+j)\delta )}{\sigma^{j(k-\ell+1)}} < \frac{C^ke^{k(m+j)\delta d}}{\sigma^{j(k-\ell+1)}}
   \end{equation}
   for any integer $j>0$. This leads to a contradiction when $j$ is large since $\sigma > e^{\frac{k}{k-\ell+1}\delta d}$.
\end{proof}

The following proposition constructs functions with good properties given any linearly independent set of ancient solutions with polynomial growth.

\begin{proposition}\label{prop: construction}
  Suppose $u_1, ...u_{2k} \in E_d$ are linearly independent. Given $\delta>0$, $m_0 > 0$ and write $\sigma = \frac{1}{2}e^{(8d+2)\delta}$, there exists integers $m>m_0$, $\ell \geq \sigma^{-1} k$ and functions $v_1,...,v_\ell$ in the linear span of $u_i$ such that for $i,j = 1,...,\ell$
  \begin{equation}
      J_{(m+1)\delta}(v_i, v_j) = \delta_{ij}, \text{ and }
  \end{equation}
  \begin{equation}
    \sigma^{-1} \leq I_{v_i}(m\delta).
  \end{equation}
\end{proposition}

\begin{proof}
  We apply Lemma \ref{lem: exp growth} with $f_i$ as in Definition \ref{def: f_i} and with $2k$, $k$, $4d+1$ and $2\sigma = e^{(8d+2)\delta}$ in place of $k$, $\ell$, $d$ and $\sigma$. Then there exist integers $m>m_0$ and $k$ integers $\alpha_1 ,..., \alpha_k \in \{1,...,2k\}$
  such that for $i = 1,...,k$,
  \begin{equation}\label{eqn: construction}
    f_{\alpha_i}((m+1)\delta) \leq 2\sigma f_{\alpha_i}((m\delta)).
  \end{equation}

  By Lemma \ref{lem: properties of f}, we can choose $m_0$ large enough so that when $m>m_0$, $f_{\alpha_i}(m\delta)> 0$. It then follows that $w_{\alpha_1, (m+1)\delta}$, ..., $w_{\alpha_k, (m+1)\delta}$ are linearly independent. Let $V$ be the linear span of these functions. On the $k$-dimensional vector space $V$, $J_{(m+1)\delta}$ defines an inner product and $J_{m\delta}$ is a positive semi-definite bilinear form. Therefore we can diagonalize $J_{m\delta}$
  with respect to $J_{(m+1)\delta}$. In particular, we can find $v_1,..., v_k\in V$
  such that $J_{(m+1)\delta}(v_i,v_j) = \delta_{ij}$ for $i,j = 1,...,k$ and $J_{m\delta}(v_i, v_j) = 0$ when $i\not = j$. Since the trace of $J_{m\delta}$ is independent of the basis, we compute it with respect to the bases $\{w_{\alpha_i, (m+1)\delta}/f_{\alpha_i}((m+1)\delta)\}$ and $\{v_i\}$.

  By Lemma \ref{lem: properties of f}(2) and inequality (\ref{eqn: construction}), we have,
  \begin{equation}\label{eqn: trace bd}
    \sum_{i=1}^k I_{v_i}(m\delta) = \sum_{i=1}^k \frac{I_{w_{\alpha_i, (m+1)\delta}}(m\delta)}{f_{\alpha_i}((m+1)\delta)} \geq \frac{f_{\alpha_i}(m\delta)}{f_{\alpha_i}((m+1)\delta)} \geq 2\sigma^{-1}k.
  \end{equation}

  Note that $0 < I_{v_i}(m\delta) \leq I_{v_i}((m+1)\delta) = 1$. Suppose $\ell$ is the number of the $v_i$'s such that
  \begin{equation}
    \sigma^{-1} \leq I_{v_i}(m\delta).
  \end{equation}
  Then by (\ref{eqn: trace bd}), $(k-\ell)\sigma^{-1} + \ell \geq 2\sigma^{-1}k$. Therefore $\ell > \sigma^{-1}k$.
\end{proof}

\section{Finite Dimensionality}

Let $P_r(t,x) = (t-r^2, t] \times B_r(x)$ be the parabolic ball in $\R \times \R^{n+1}$. We will need the following parabolic mean value inequality. (see, e.g. \cite{L96} Theorem 6.17). Suppose $(\partial_t - L)u = 0$ on $P_r(t,x)$, then
\begin{equation}\label{eqn: mean val ineq}
  |u(t,x)|^2 \leq \frac{C}{r^{n+3}}\int_{P_r(t,x)} u^2
\end{equation}
for some constant $C = C(n,\lambda, \Lambda)$.

\begin{proposition}\label{prop: dimension bound}
  Fix $0 < \delta \leq 1$, $a>0$ and suppose $v_1$, ...,$v_\ell\in E_d$ are $J_{a+\delta}$-orthonormal. Then
  \begin{equation}
    \sum_{i=1}^\ell I_{v_i}(a) \leq C \delta^{-(n+2)}
  \end{equation}
  for some constant $C = C(n, \lambda, \Lambda, V_0)$.
\end{proposition}

\begin{proof}
  Let $V$ be the linear span of $v_1$, ... $v_\ell$. For each $(t, x)\in Q_{a+\delta}$, we define
  \begin{equation}
    K(t,x) = \sum_{i=1}^\ell |v_i|^2(t,x).
  \end{equation}
  Since each $v_i \in L^2_{loc}(\R^- \times \Omega)$, $K(t,x)$ is finite by the mean value inequality. Moreover, it is the trace of a bilinear form $(v, w) \to v(t, x)w(t, x)$ on $V$. Therefore, we can diagonalize it by an orthogonal change of basis. Let $w_1$, ..., $w_\ell$ are $J_{a+\delta}$-orthonormal basis that diagonalizes $K(t, x)$, then there exists at most one $w_i$ that is not zero at $(t,x)$. Without the loss of generallity, let it be $w_1$. Since $K(t,x)$ is the trace, we have
  \begin{equation}
    K(t,x) = w_1^2(t,x).
  \end{equation}

  Now we use mean value inequality to bound $K$ pointwise. Note that if $\rho >r >0$ and $(t,x)\in Q_r$, then $P_{(\rho-r)}(t,x)\cap (\R^-\times \Omega) \subset Q_\rho$. We extend $w_1$ outside $\R^- \times \Omega$ by zero and apply the mean value inequality (\ref{eqn: mean val ineq}) to $w_1$ at $(t,x) \in Q_r$. We get
  \begin{equation}
    \begin{aligned}
      K(t,x) = w_1^2(t,x) &\leq \frac{C}{(a+\delta - r)^{n+3}}\int_{P_{a+\delta - r}(t,x)} w_1^2 \\
      &\leq \frac{C}{(a+\delta - r)^{n+3}}\int_{Q_{a+\delta}} w_1^2 \\
      &= C(a+\delta - r)^{-(n+3)}.
    \end{aligned}
  \end{equation}

  One one hand, integrating $K(t,x)$ over $Q_{a} \setminus Q_\frac{a }{2}$, we obtain
  \begin{equation}\label{eqn: bound annulus}
    \begin{aligned}
      \sum_{i = 1}^\ell \int_{Q_{a} \setminus Q_\frac{a }{2}} v_i^2 &= \int_{Q_{a}\setminus Q_\frac{a }{2}} K(t,x) \\
      &\leq CV_0 \int_\frac{a }{2}^a (a+\delta - r)^{-(n+3)} r^2 \,dr \\
      &\leq C'a^2 \delta^{-(n+2)}.
    \end{aligned}
  \end{equation}
  On the other hand, Lemma \ref{lem: exp growth} asserts that
  \begin{equation}\label{eqn: bound whole thing}
    \int_{Q_a} v_i^2 \leq \frac{C }{a^2} \int_{Q_a\setminus Q_{a/2}} v_i^2.
  \end{equation}

  The proposition follows by combining (\ref{eqn: bound annulus}) and (\ref{eqn: bound whole thing}).
\end{proof}

Theorem \ref{main thm 2} is then an immediate consequence of Proposition \ref{prop: construction} and \ref{prop: dimension bound}.
\begin{proof}(of Theorem \ref{main thm 2})
  Let $u_1$, ..., $u_{2k} \in E_d$ be any set of linearly independent ancient solutions with exponential growth. Set $\delta = \frac{1}{d}$ in Proposition \ref{prop: construction} and $\sigma = \frac{1}{2}e^{8+\frac{1}{2d}} \leq e^{10}$ since $d\geq 1$. There then exists $m>0$, $\ell > \sigma^{-1}k \geq e^{-10}k$
  and $J_{(m+1)\delta}$-orthonormal functions $v_1, ..., v_\ell$ such that for $i=1,... \ell$
  \begin{equation}
    e^{-10} \leq \sigma^{-1} \leq I_{v_i}(m\delta).
  \end{equation}
  Therefore,
  \begin{equation}\label{eqn: bound dim 1}
    \sum_{i=1}^\ell I_{v_i}(m\delta) \geq e^{-10}\ell \geq e^{-20}k.
  \end{equation}
  But Proposition \ref{prop: dimension bound} implies that
  \begin{equation}\label{eqn: bound dim 2}
    \sum_{i=1}^\ell I_{v_i}(m\delta) \leq C\delta^{-(n+2)} = Cd^{n+2}.
  \end{equation}

  Combining (\ref{eqn: bound dim 1}) and (\ref{eqn: bound dim 2}), we have $\dim E_d \leq Cd^{n+2}$ for some constant $C = C(n,\lambda, \Lambda, V_0)$.
\end{proof}


\end{document}